\documentclass[10pt]{amsart}
\usepackage{amsfonts}
\usepackage{amsthm}
\usepackage{amsmath}
\usepackage{graphicx}
\usepackage{amscd,amssymb,amsthm}

\newtheorem{theorem}{Theorem}
\newtheorem{lemma}{Lemma}

\newcounter{minutes}
\divide\time by 60
\newcounter{hours}
\multiply\time by 60 \addtocounter{minutes}{-\time}

\title{Tur\'an type inequalities for regular Coulomb wave functions}

\author[\'A. Baricz]{\'Arp\'ad Baricz}
\address{Institute of Applied Mathematics, \'Obuda University, Budapest, Hungary}
\address{Department of Economics, Babe\c{s}-Bolyai University, Cluj-Napoca, Romania}
\email{bariczocsi@yahoo.com}

\keywords{Coulomb wave functions; Mittag-Leffler expansion; Tur\'an, Mitrinovi\'c-Adamovi\'c, Wilker type
inequality; Coulomb zeta functions; complete monotonicity; interlacing property of zeros.}

\subjclass[2010]{Primary 33C15, Secondary 26D07, 39C05.}

\thanks{$^{\bigstar}$The research of \'A. Baricz was supported by the J\'anos Bolyai Research Scholarship of the Hungarian Academy of Sciences. The author is very grateful to Mourad E.H. Ismail for suggesting the study on the regular Coulomb wave function and also for his kind hospitality during the author's visit at Department of Mathematics of the City University of Hong Kong in September 2011. The author is also grateful to his colleague \'Agoston R\'oth for the discussions about the growth order of the regular Coulomb wave function.}

\begin{document}

\def\thefootnote{}
\footnotetext{ \texttt{File:~\jobname .tex,
          printed: \number\year-0\number\month-\number\day,
          \thehours.\ifnum\theminutes<10{0}\fi\theminutes}
} \makeatletter\def\thefootnote{\@arabic\c@footnote}\makeatother

\maketitle

\begin{abstract}
Tur\'an, Mitrinovi\'c-Adamovi\'c and Wilker type inequalities are deduced for regular Coulomb wave functions. The
proofs are based on a Mittag-Leffler expansion for the regular Coulomb wave function, which may be of independent interest.
Moreover, some complete monotonicity results concerning the Coulomb zeta functions and some interlacing properties
of the zeros of Coulomb wave functions are given.
\end{abstract}

\section{Introduction}
\setcounter{equation}{0}

The Coulomb wave function, which bears the name of the famous French physicist Charles Augustin de Coulomb (best known for his law describing the electrostatic interaction between electrically charged particles), is a solution of the Coulomb wave equation (or radial Schr\"odinger equation in the Coulomb potential) and it is used to describe the behavior of charged particles in a Coulomb potential. There is an extensive literature concerning the computation of the Coulomb wave function values, however, the zeros and other analytical properties have not been studied in detail. For more details we refer to the papers \cite{ikebe,miyazaki} and to the references therein. We mention that recently, an important study on the Coulomb wave function was made by \v{S}tampach and \v{S}\v{t}ov\'\i\v{c}ek \cite{stampach}. In this paper we present some new results on the Coulomb wave function, which may be useful for people working in special functions and mathematical physics. Our present paper belongs to the rich literature about Tur\'an type inequalities on orthogonal polynomials and special functions, named after the Hungarian mathematician Paul Tur\'an, and can be interpreted as the generalization of some of the results on Bessel functions of the first kind, obtained by Sz\'asz \cite{szasz1,szasz2}. The paper is organized as follows: the next section is divided into four subsections and contains some Tur\'an, Mitrinovi\'c-Adamovi\'c and Wilker type inequalities for the regular Coulomb wave function. The key tool in the proofs is a Mittag-Leffler expansion for the regular Coulomb wave function, which may be of independent interest. We also deduce some complete monotonicity results for the Coulomb zeta functions, which are defined by using the real zeros of the Coulomb wave functions. By using the Hadamard factorization of the Coulomb wave functions we also present some interlacing properties of the zeros of the Coulomb wave functions.

\section{Properties of the regular Coulomb wave functions}
\setcounter{equation}{0}

In this section our aim is to present the main results of this paper about the regular Coulomb wave function together with their proofs. The section is divided into four subsections.

\subsection{Tur\'an type inequalities for regular Coulomb wave functions} In order to obtain the main results of this subsection we use a Mittag-Leffler expansion for the regular Coulomb wave function together with the recurrence relations, and a result of Ross
\cite{ross}. As we can see below the second main result of this subsection
is a natural extension of a well-known result of Sz\'asz \cite{szasz1,szasz2} for Bessel functions of the
first kind. The next result, which may be of independent interest, is an immediate consequence of a result of Wimp \cite{wimp} concerning confluent hypergeometric functions and it was recently rediscovered by
\v{S}tampach and \v{S}\v{t}ov\'\i\v{c}ek \cite{stampach}, by using a different method. In both papers \cite{stampach,wimp}
a new class of orthogonal polynomials associated with regular Coulomb wave functions is introduced. These polynomials play a role analogous to that the Lommel polynomials have in the theory of Bessel functions of the first kind. However, it is worth to mention that Wimp's approach \cite{wimp} is based on inversion of Stieltjes transforms, while \v{S}tampach and \v{S}\v{t}ov\'\i\v{c}ek \cite{stampach} used the eigenvalues of some Jacobi matrices.

\begin{lemma}
Let $\rho,\eta\in\mathbb{R}$ and let $L>-3/2,$ $L\neq-1$ if $\eta\neq0$ and $L>-3/2$ if $\eta=0.$ Then the next Mittag-Leffler expansion is valid
\begin{equation}\label{Comitt}
\frac{F_{L+1}(\eta,\rho)}{F_{L}(\eta,\rho)}=\frac{L+1}{\sqrt{(L+1)^2+\eta^2}}\sum_{n\geq 1}\left[\frac{\rho}{x_{L,\eta,n}(x_{L,\eta,n}-\rho)}+\frac{\rho}{y_{L,\eta,n}(y_{L,\eta,n}-\rho)}\right],
\end{equation}
where $x_{L,\eta,n}$ and $y_{L,\eta,n}$ are the $n$th positive and negative zeros of the Coulomb wave
function $F_L(\eta,\rho).$
\end{lemma}

\begin{proof}
Let ${}_1F_1$ denotes the Kummer confluent hypergeometric function. It is known that
$$F_{L}(\eta,\rho)=C_{L}(\eta)\rho^{L+1}e^{-\mathrm{i}\rho}{}_1F_1(L+1-\mathrm{i}\eta,2L+2;2\mathrm{i}\rho),$$
where $$C_L(\eta)=\frac{2^Le^{-\frac{\pi\eta}{2}}\left|\Gamma(L+1+\mathrm{i}\eta)\right|}{\Gamma(2L+2)}.$$
By using the next result of Wimp \cite[p. 892]{wimp} for $c=2L+2,$ $\kappa=\eta$ and $z=1/\rho$
$$\frac{{}_1F_1\left(\frac{c}{2}+1-\mathrm{i}\kappa,c+2;\frac{2\mathrm{i}}{z}\right)}{{}_1F_1\left(\frac{c}{2}-\mathrm{i}\kappa,c;\frac{2\mathrm{i}}{z}\right)}=
\frac{c^2(c+1)}{c^2+4\kappa^2}\sum_{k\in\mathbb{Z}\setminus\{0\}}z_k^{-2}\frac{z}{z-z_k^{-1}},$$
where $\kappa,z\in\mathbb{R},$ $c>-1$ and $z_k,$ $k\in\mathbb{Z}\setminus\{0\},$ are the zeros of the function ${}_1F_1(c/2-\mathrm{i}\kappa,c;2\mathrm{i}z),$
it follows that
$$\frac{F_{L+1}(\eta,\rho)}{F_{L}(\eta,\rho)}=\frac{C_{L+1}(\eta)}{C_L(\eta)}\frac{(L+1)^2(2L+3)}{(L+1)^2+\eta^2}\sum_{n\geq 1}\left[\frac{\rho}{x_{L,\eta,n}(x_{L,\eta,n}-\rho)}+\frac{\rho}{y_{L,\eta,n}(y_{L,\eta,n}-\rho)}\right],
$$
which by means of the relation \cite[p. 538]{abra} $L(2L+1)C_L(\eta)=\sqrt{L^2+\eta^2}C_{L-1}(\eta)$ yields \eqref{Comitt}. We note that in the above formula of Wimp \cite[p. 892]{wimp} instead of the correct expression $K=c^2(c+1)/(c^2+4\kappa^2)$ it was used $K=c^2(c+1)/(c^2/4+\kappa^2),$ and instead of the correct argument $2\mathrm{i}/z$ it was $\mathrm{i}/z.$ This can be verified by using the fact that when $\eta=0$ the Coulomb wave function reduces to Bessel function of the first kind, and by using the Mittag-Leffler expansion for Bessel functions of the first kind and the first Rayleigh sum of zeros of Bessel functions we would have contradiction.

Another way to obtain \eqref{Comitt} is to consider the Hadamard infinite product expansion \cite[eq. 76]{stampach}
\begin{equation}\label{infprod}F_{L}(\eta,\rho)=C_L(\eta)\rho^{L+1}e^{\frac{\eta\rho}{L+1}}
\prod_{n\geq1}\left(1-\frac{\rho}{\rho_{L,\eta,n}}\right)e^{\frac{\rho}{\rho_{L,\eta,n}}},\end{equation}
where $\rho_{L,\eta,n}$ is the $n$th zero of the Coulomb wave function. Logarithmic derivation yields
$$\frac{F_{L}'(\eta,\rho)}{F_L(\eta,\rho)}=\frac{L+1}{\rho}+\frac{\eta}{L+1}-\sum_{n\geq 1}\frac{\rho}{\rho_{L,\eta,n}(\rho_{L,\eta,n}-\rho)},$$
which in view of the recurrence relation \cite[p. 539]{abra}
\begin{equation}\label{reccou2}(L+1)F_L'(\eta,\rho)=\left[\frac{(L+1)^2}{\rho}+\eta\right]
F_{L}(\eta,\rho)-\sqrt{(L+1)^2+\eta^2}F_{L+1}(\eta,\rho)\end{equation}
yields
$$\frac{F_{L+1}(\eta,\rho)}{F_{L}(\eta,\rho)}=\frac{L+1}{\sqrt{(L+1)^2+\eta^2}}\sum_{n\geq 1}\frac{\rho}{\rho_{L,\eta,n}(\rho_{L,\eta,n}-\rho)}.$$
Now, taking into account that the zeros $\rho_{L,\eta,n}$ can be separated into positive and negative zeros, the proof of \eqref{Comitt} is done.
\end{proof}

It is worth to mention that if $\eta=0,$ then \eqref{Comitt} reduces to the next well-known Mittag-Leffler expansion
$$\frac{F_{L+1}(0,\rho)}{F_{L}(0,\rho)}=\frac{J_{L+3/2}(\rho)}{J_{L+1/2}(\rho)}=\sum_{n\geq1}\frac{2\rho}{j_{L+1/2,n}^2-\rho^2},$$
where $L>-3/2,$ $J_{L}$ stands for the Bessel function of the first kind of order $L$ and $j_{L,n}$ is the $n$th positive zero of the
Bessel function $J_L.$ Here we used that for each natural $n$ we have $x_{L,0,n}=-y_{L,0,n}=j_{L+1/2,n},$ that is, the corresponding negative and positive zeros of the Bessel function of the first kind are symmetric with respect to the origin.

Now, we are ready to present the first set of results concerning the Tur\'an type inequalities for the regular Coulomb wave function. Three kind of Tur\'anians are considered and the results are mainly based on the Mittag-Leffler expansion \eqref{Comitt}. Our first main result is the following theorem.

\begin{theorem}
The following assertions are true:
\begin{enumerate}

\item[\bf a.] If $L,\eta>0,$ $0<\rho<L(L+1)/\eta,$
$\rho<x_{L,\eta,1}$ or $-3/2<L<-1,$ $\eta>0,$ $0<\rho<L(L+1)/\eta,$
$\rho<x_{L,\eta,1}$ or $\eta\leq0,$ $L\geq0$ and
$0<\rho<x_{L,\eta,1},$ then
\begin{equation}\label{coturan1}
F_{L}^2(\eta,\rho)-F_{L-1}(\eta,\rho)F_{L+1}(\eta,\rho)\geq0.
\end{equation}

\item[\bf b.] If $L,\eta>0,$ $L(L+1)/\eta\leq\rho<x_{L-1,\eta,1}$ or $-3/2<L<-1,$ $\eta>0,$ $L(L+1)/\eta\leq\rho<x_{L-1,\eta,1}$ or $-1<L<0,$ $\eta<0,$ $L(L+1)/\eta\leq\rho<x_{L-1,\eta,1}$ then
$$\frac{\sqrt{L^2+\eta^2}}{L}F_{L}^2(\eta,\rho)-
\frac{\sqrt{(L+1)^2+\eta^2}}{L+1}F_{L-1}(\eta,\rho)F_{L+1}(\eta,\rho)\geq0.$$

\item[\bf c.] If $L>-1,$ $\eta\in\mathbb{R},$ $\rho^2\leq (L^3+1)/(L^2+\eta^2),$ ${\eta}/({L(L+1)})-
{1}/{\rho}>0$ and $0<\rho<x_{L-1,\eta,1},$ then
$$F_{L}^2(\eta,\rho)-
\frac{\sqrt{L^2+\eta^2}\sqrt{(L+1)^2+\eta^2}}{L(L+1)}F_{L-1}(\eta,\rho)
F_{L+1}(\eta,\rho)\geq0.$$
\end{enumerate}
\end{theorem}

\begin{proof}
{\bf a.} By using the recurrence relation
\cite[p. 539]{abra}
\begin{equation}\label{reccou1}LF_L'(\eta,\rho)=\sqrt{L^2+\eta^2}F_{L-1}(\eta,\rho)-
\left(\frac{L^2}{\rho}+\eta\right)F_{L}(\eta,\rho)\end{equation}
and \eqref{reccou2}, we obtain
$$\frac{{}_1\Delta_{L,\eta}(\rho)}{F_{L}^2(\eta,\rho)}=a_{L,\eta}(\rho)-
b_{L,\eta}(\rho)\frac{F_{L}'(\eta,\rho)}{F_{L}(\eta,\rho)}+
c_{L,\eta}\left[\frac{F_{L}'(\eta,\rho)}{F_{L}(\eta,\rho)}\right]^2,$$
where
$$a_{L,\eta}(\rho)=1-\frac{\left(\frac{L^2}{\rho}+\eta\right)\left[\frac{(L+1)^2}{\rho}+\eta\right]}
{\sqrt{L^2+\eta^2}\sqrt{(L+1)^2+\eta^2}},\ b_{L,\eta}(\rho)=\frac{\frac{L(L+1)}{\rho}-\eta}{\sqrt{L^2+\eta^2}\sqrt{(L+1)^2+\eta^2}},$$
$$c_{L,\eta}=\frac{L(L+1)}{\sqrt{L^2+\eta^2}\sqrt{(L+1)^2+\eta^2}}$$
and ${}_1\Delta_{L,\eta}(\rho)$ stands for the Tur\'an expression,
defined by
$${}_1\Delta_{L,\eta}(\rho)=F_{L}^2(\eta,\rho)-F_{L-1}(\eta,\rho)F_{L+1}(\eta,\rho).$$
Now, taking into account that the Coulomb wave function is a
particular solution of the Coulomb differential equation \cite[p.
538]{abra}
\begin{equation}\label{eqcou}w''(\rho)+\left[1-\frac{2\eta}{\rho}-
\frac{L(L+1)}{\rho^2}\right]w(\rho)=0,\end{equation} we get
\begin{equation}\label{eqdiffercou}\left[\frac{F_{L}'(\eta,\rho)}{F_{L}(\eta,\rho)}\right]^2=
\frac{L(L+1)}{\rho^2}+\frac{2\eta}{\rho}-1-
\left[\frac{F_{L}'(\eta,\rho)}{F_{L}(\eta,\rho)}\right]',\end{equation}
which in turn implies that
$$\frac{{}_1\Delta_{L,\eta}(\rho)}{F_{L}^2(\eta,\rho)}=d_{L,\eta}(\rho)-b_{L,\eta}(\rho)
\frac{F_{L}'(\eta,\rho)}{F_{L}(\eta,\rho)}-
c_{L,\eta}\left[\frac{F_{L}'(\eta,\rho)}{F_{L}(\eta,\rho)}\right]',$$
where
$$d_{L,\eta}(\rho)=1-\frac{L(L+1)+\frac{\eta}{\rho}+\eta^2}
{\sqrt{L^2+\eta^2}\sqrt{(L+1)^2+\eta^2}}.$$
Moreover, by using the recurrence relation \eqref{reccou2} and the
Mittag-Leffler expansion \eqref{Comitt}, it follows
$$\frac{F_{L}'(\eta,\rho)}{F_{L}(\eta,\rho)}=\frac{L+1}{\rho}+\frac{\eta}{L+1}-\sum_{n\geq1}
\left[\frac{\rho}{x_{L,\eta,n}(x_{L,\eta,n}-\rho)}+
\frac{\rho}{y_{L,\eta,n}(y_{L,\eta,n}-\rho)}\right]$$ and
\begin{equation}\label{logder}\left[\frac{F_{L}'(\eta,\rho)}{F_{L}(\eta,\rho)}\right]'
=-\frac{L+1}{\rho^2}-\sum_{n\geq1}
\left[\frac{1}{(x_{L,\eta,n}-\rho)^2}+\frac{1}{(y_{L,\eta,n}-\rho)^2}\right].\end{equation}
Consequently we have
\begin{align*}\frac{{}_1\Delta_{L,\eta}(\rho)}{F_{L}^2(\eta,\rho)}=e_{L,\eta}+b_{L,\eta}&(\rho)\sum_{n\geq1}
\left[\frac{\rho}{x_{L,\eta,n}(x_{L,\eta,n}-\rho)}+
\frac{\rho}{y_{L,\eta,n}(y_{L,\eta,n}-\rho)}\right]
\\&+c_{L,\eta}\sum_{n\geq1}
\left[\frac{1}{(x_{L,\eta,n}-\rho)^2}+\frac{1}{(y_{L,\eta,n}-\rho)^2}\right],\end{align*}
where
$$e_{L,\eta}=1-\frac{L\sqrt{(L+1)^2+\eta^2}}{(L+1)\sqrt{L^2+\eta^2}}.$$
Note that for all $L\geq0$ or $-3/2<L<-1$ and $\eta\in\mathbb{R}$ we have
$c_{L,\eta}\geq0$ and $e_{L,\eta}\geq0.$ Thus
${}_1\Delta_{L,\eta}(\rho)$ is positive if $L,\eta>0,$
$0<\rho<L(L+1)/\eta,$ $\rho<x_{L,\eta,1}$ or if $-3/2<L<-1,$ $\eta>0,$
$0<\rho<L(L+1)/\eta,$ $\rho<x_{L,\eta,1}$ or if $\eta\leq0,$
$L\geq0$ and $0<\rho<x_{L,\eta,1}.$

{\bf b.} By using the recurrence relations \eqref{reccou1} and
\eqref{reccou2} we obtain
$$F_{L+1}'(\eta,\rho)F_{L}(\eta,\rho)-F_{L}'(\eta,\rho)F_{L+1}(\eta,\rho)=
{}_2\Delta_{L+1,\eta}(\rho)-\left[\frac{\eta}{(L+1)(L+2)}-
\frac{1}{\rho}\right]F_{L}(\eta,\rho)F_{L+1}(\eta,\rho),$$ where
$${}_2\Delta_{L,\eta}(\rho)=\frac{\sqrt{L^2+\eta^2}}{L}F_{L}^2(\eta,\rho)-
\frac{\sqrt{(L+1)^2+\eta^2}}{L+1}F_{L-1}(\eta,\rho)F_{L+1}(\eta,\rho).$$
On the other hand, according to \cite[Lemma 2.4]{miyazaki} we have
$$\rho^2\frac{\sqrt{(L+1)^2+\eta^2}}{L+1}\left[F_{L+1}'(\eta,\rho)F_{L}(\eta,\rho)-
F_{L}'(\eta,\rho)F_{L+1}(\eta,\rho)\right]=\sum_{n\geq1}(2L+2n+1)F_{L+n}^2(\eta,\rho).
$$
From this we obtain that
$$\frac{{}_2\Delta_{L,\eta}(\rho)}{F_{L-1}^2(\eta,\rho)}\geq\left[\frac{\eta}{L(L+1)}-
\frac{1}{\rho}\right]\frac{F_{L}(\eta,\rho)}{F_{L-1}(\eta,\rho)}$$
and by using the Mittag-Leffler expansion \eqref{Comitt}, the right-hand
side of the above inequality is positive if $L,\eta>0$ and
$L(L+1)/\eta\leq\rho<x_{L-1,\eta,1}$ or if $-3/2<L<-1,$ $\eta>0$ and
$L(L+1)/\eta\leq\rho<x_{L-1,\eta,1}$ or if $-1<L<0,$ $\eta<0,$ $L(L+1)/\eta\leq\rho<x_{L-1,\eta,1}.$

{\bf c.} Observe that \eqref{logder} implies that for all
$\eta,\rho\in\mathbb{R},$ $\rho\neq0$ and $L\geq-1$ we have
$$D_{L,\eta}(\rho)=F_{L}''(\eta,\rho)F_{L}(\eta,\rho)-F_{L}'(\eta,\rho)F_{L}'(\eta,\rho)\leq0.$$
Now, by using the recurrence relations \eqref{reccou1} and
\eqref{reccou2} and also the fact that $F_L(\eta,\rho)$ satisfies
the Coulomb differential equation \eqref{eqcou}, we obtain
$$D_{L,\eta}(\rho)=f_{L,\eta}(\rho)F_{L}^2(\eta,\rho)+
\frac{1}{c_{L,\eta}}F_{L-1}(\eta,\rho)F_{L+1}(\eta,\rho)+
\left[\frac{\eta}{L(L+1)}-\frac{1}{\rho}\right]F_{L-1}(\eta,\rho)F_{L}(\eta,\rho),$$
where
$$f_{L,\eta}(\rho)=\frac{L}{\rho^2}-1-\frac{\eta^2}{L^2}.$$
If $L>-1,$ $\eta\in\mathbb{R}$ and $(L^3+1)/(L^2+\eta^2)\geq\rho^2,$ then we have that
$f_{L,\eta}(\rho)\geq-1$ and consequently we have
$$0\geq D_{L,\eta}(\rho)\geq -{}_3\Delta_{L,\eta}(\rho)+
\left[\frac{\eta}{L(L+1)}-\frac{1}{\rho}\right]F_{L-1}(\eta,\rho)F_{L}(\eta,\rho),$$
where
$${}_3\Delta_{L,\eta}(\rho)=F_{L}^2(\eta,\rho)-
\frac{\sqrt{L^2+\eta^2}\sqrt{(L+1)^2+\eta^2}}{L(L+1)}F_{L-1}(\eta,\rho)F_{L+1}(\eta,\rho).$$
But the above inequality is equivalent to
$$\frac{{}_3\Delta_{L,\eta}(\rho)}{F_{L-1}^2(\eta,\rho)}\geq\left[\frac{\eta}{L(L+1)}-
\frac{1}{\rho}\right]\frac{F_{L}(\eta,\rho)}{F_{L-1}(\eta,\rho)}$$
and by using again the Mittag-Leffler expansion \eqref{Comitt}, the
right-hand side of the above inequality is positive if ${\eta}/({L(L+1)})-
{1}/{\rho}>0$ and $0<\rho<x_{L-1,\eta,1}.$ With this the proof is
complete.
\end{proof}

Now, let us consider the notations
$$B_{L,\eta}(\rho)=\frac{L\sqrt{(L+1)^2+\eta^2}}{(2L+1)\left[\frac{L(L+1)}{\rho}+\eta\right]}\ \ \ \mbox{and}\ \ \ C_{L,\eta}(\rho)=\frac{(L+1)\sqrt{L^2+\eta^2}}{(2L+1)\left[\frac{L(L+1)}{\rho}+\eta\right]}.$$
In what follows we show that if $L\geq0,$ $\eta\leq0,$ then the
restriction $\rho<x_{L,\eta,1}$ in the Tur\'an type inequality
\eqref{coturan1} can be removed. Moreover, we show that in this case the
inequality \eqref{coturan1} can be improved.

\begin{theorem}\label{thturan}
If $n\in\{0,1,\dots\}$ and $L\geq-3/2,$ $L\neq-1,$ $\rho>0,$ $\eta\in\mathbb{R},$ $\eta\neq0$ or $L>-3/2,$ $\rho>0$ and $\eta=0,$ then
\begin{align}\label{identcou}
&F_{L+n}^2(\eta,\rho)-F_{L+n-1}(\eta,\rho)F_{L+n+1}(\eta,\rho)=-\frac{\Theta
C_{L+n,\eta}(\rho)}{C_{L+n,\eta}(\rho)}F_{L+n}^2(\eta,\rho)\\&-\sum_{i=1}^{\infty}\frac{B_{L+n+1,\eta}(\rho)B_{L+n+2,\eta}(\rho)\dots
B_{L+n+i+1,\eta}(\rho)}{C_{L+n,\eta}(\rho)C_{L+n+1,\eta}(\rho)\dots
C_{L+n+i,\eta}(\rho)}\Theta(B_{L+n+i-1,\eta}(\rho)C_{L+n+i,\eta}(\rho))F_{L+n+i}^2(\eta,\rho)\nonumber,
\end{align}
where $\Theta$ is the forward difference operator defined by $\Theta
A_n=A_{n+1}-A_n.$

In particular, for all $L\geq 0,$ $\eta\leq0$ and $\rho>0$ the
following sharp Tur\'an type inequality is valid
\begin{equation}\label{coturan4}
F_{L}^2(\eta,\rho)-F_{L-1}(\eta,\rho)F_{L+1}(\eta,\rho)\geq
\left[1-\frac{L(2L+1)\sqrt{(L+1)^2+\eta^2}}{(L+1)(2L+3)\sqrt{L^2+\eta^2}}\right]F_{L}^2(\eta,\rho).\end{equation}
\end{theorem}

It is important to mention here that when $\eta=0$ the Tur\'an
type inequalities \eqref{coturan1} and \eqref{coturan4} reduce to
known results of Sz\'asz \cite{szasz1}. More precisely, since
\cite[p. 542]{abra}
$F_L(0,\rho)=\sqrt{\frac{\pi\rho}{2}}J_{L+1/2}(\rho),$ the Tur\'an type inequalities \eqref{coturan1} and
\eqref{coturan4} for $\eta=0$ and $L+1/2=\nu$ become
$$J_{\nu}^2(\rho)-J_{\nu-1}(\rho)J_{\nu+1}(\rho)\geq 0,$$
$$J_{\nu}^2(\rho)-J_{\nu-1}(\rho)J_{\nu+1}(\rho)\geq \frac{1}{\nu+1}J_{\nu}^2(\rho),$$
where $\nu\geq1/2$ and $\rho>0.$ For more details on Tur\'an type
inequalities for Bessel functions and other generalizations we refer
to the papers \cite{BP,bustoz,JB,karlin,rao,patrick,skov,szasz2,thiru}
and the references therein.

The proof of the above theorem is based on the next result of Ross
\cite[Theorem 3]{ross}.

\begin{lemma}\label{leross}
Let $I$ be an interval and let $\{y_n\}_{n\geq0}$ be a sequence of
functions of real variable $x$, which is uniformly bounded in $n$
for each $x\in I.$ If these functions satisfy
$$y_n(x)=B_ny_{n+1}(x)+C_ny_{n-1}(x),$$
where $B_n$ and $C_n$ are functions of $x,$ $x\in I,$ with the
property that $C_n(x)\neq0,$ $B_n(x)\to0$ and
$\prod_{i=1}^n|B_i(x)/C_i(x)|$ converges as $n\to\infty$ for all
$x\in I,$ then
\begin{equation}\label{eqross}
y_n^2(x)-y_{n-1}(x)y_{n+1}(x)=-\frac{\Theta
C_n}{C_n}y_n^2(x)-\sum_{i=1}^{\infty}\frac{B_{n+1}B_{n+2}\dots
B_{n+i+1}}{C_nC_{n+1}\dots
C_{n+i}}\Theta(B_{n+i-1}C_{n+i})y_{n+i}^2(x).
\end{equation}
\end{lemma}

For reader's convenience we note here that in formula (i) of
\cite[p. 28]{ross} the expression $B_ny_n$ should be written as
$B_{n+1}y_n,$ and in the main formula of \cite[Theorem 3]{ross} the
expression $B_{n+i-1}$ should be written as $B_{n+i+1},$ just like
in \eqref{eqross}.

\begin{proof}[Proof of Theorem \ref{thturan}] In order to deduce the infinite sum representation of
the Tur\'anian of the Coulomb wave functions in Theorem \ref{thturan} we shall use Lemma
\ref{leross}. According to the recurrence relation \cite[p. 539]{abra}
$$B_{L,\eta}(\rho)F_{L+1}(\eta,\rho)=F_{L}(\eta,\rho)-C_{L,\eta}(\rho)F_{L-1}(\eta,\rho)$$
we have
$$F_{L+n}(\eta,\rho)=B_{L+n,\eta}(\rho)F_{L+n+1}(\eta,\rho)+C_{L+n,\eta}(\rho)F_{L+n-1}(\eta,\rho).$$
Observe that when $n\in\{0,1,\dots\}$ for $L>-3/2,$ $L\neq-1,$ $\rho\in\mathbb{R}$ and $\eta\in\mathbb{R},$ $\eta\neq0,$ or $L>-3/2,$ $\rho\in\mathbb{R}$ and $\eta=0$ we have $C_{L+n,\eta}(\rho)\neq0$ and
$B_{L+n,\eta}(\rho)\to0$ as $n\to\infty.$ Moreover, the product
$$\prod_{i=1}^n\frac{B_{L+i,\eta}(\rho)}{C_{L+i,\eta}(\rho)}=
\prod_{i=1}^n\frac{(L+i)\sqrt{(L+i+1)^2+\eta^2}}{(L+i+1)\sqrt{(L+i)^2+\eta^2}}=
\sqrt{\frac{1+\frac{\eta^2}{(L+n+1)^2}}{1+\frac{\eta^2}{(L+1)^2}}}$$
converges as $n\to\infty$ for all $L>-3/2,$ $L\neq-1,$ $\rho\in\mathbb{R}$ and
$\eta\in\mathbb{R}.$ We just need to check the uniform boundedness
of the Coulomb wave function with respect to $L+n.$ For this we use
the asymptotic relation $F_{L}(\eta,\rho)\sim C_L(\eta)\rho^{L+1}$
as $L\to\infty.$ Note that according to \cite[p. 538]{abra} and
\cite[p. 43]{nishi} for $L$ positive integer we have
$$C_L(\eta)=\frac{2^Le^{-\frac{\pi\eta}{2}}\left|\Gamma(L+1+\mathrm{i}\eta)\right|}{\Gamma(2L+2)}=
\left\{\begin{array}{lc}\frac{2^L}{(2L+1)!}\sqrt{\frac{2\pi\prod_{k=0}^L(k^2+\eta^2)}{\eta(e^{2\pi\eta}-1)}},&
\mbox{if}\ \ \ \eta\neq0\\\frac{2^LL!}{(2L+1)!},& \mbox{if}\ \ \
\eta=0\end{array}\right..$$ Thus, by using the infinite product
representation of the hyperbolic sine function \cite[p. 85]{abra} we
get that for fixed $\eta\in\mathbb{R}$ and $\rho>0$
$$C_L(\eta)\rho^{L+1}\to C_{L}(0)\rho^{L+1}\sqrt{\frac{2\sinh(\pi\eta)}{e^{2\pi\eta}-1}}=
\frac{\sqrt{\pi}\rho^{L+1}}{2^{L-1}\Gamma\left(L+\frac{3}{2}\right)}
e^{-\frac{\pi\eta}{2}}\to0\ \ \ \mbox{as}\ \ \ L\to \infty,$$ and
consequently
$$C_{L+n}(\eta)\rho^{L+n+1}\to0\ \ \ \mbox{as}\ \ \ n\to \infty.$$
Thus, applying \eqref{eqross}, the proof of \eqref{identcou} is
complete.

Now, let us focus on the Tur\'an type inequality
\eqref{coturan4}. If we choose $n=0$ in \eqref{identcou}, then we
obtain
$$
{}_1\Delta_{L,\eta}(\rho)=\left(1-\frac{C_{L+1,\eta}(\rho)}{C_{L,\eta}(\rho)}\right)F_{L}^2(\eta,\rho)-
\sum_{i=1}^{\infty}\frac{B_{L+1,\eta}(\rho)\dots
B_{L+i+1,\eta}(\rho)}{C_{L,\eta}(\rho)\dots
C_{L+i,\eta}(\rho)}\Theta(B_{L+i-1,\eta}(\rho)C_{L+i,\eta}(\rho))F_{L+i}^2(\eta,\rho).
$$
In what follows we show that
\begin{equation}\label{ineqross}\Theta(B_{L+i-1,\eta}(\rho)C_{L+i,\eta}(\rho))=
B_{L+i,\eta}(\rho)C_{L+i+1,\eta}(\rho)-B_{L+i-1,\eta}(\rho)C_{L+i,\eta}(\rho)\leq0\end{equation}
for all $L\geq0,$ $\eta\leq0,$ $\rho>0$ and $i\in\{1,2,\dots\}.$
Observe that the above inequality can be written as
$$\frac{(L+i)(L+i+2)\left((L+i+1)^2+\eta^2\right)}{(2L+2i+3)\left((L+i+1)(L+i+2)+\rho\eta\right)}
\leq\frac{(L+i-1)(L+i+1)\left((L+i)^2+\eta^2\right)}{(2L+2i-1)\left((L+i-1)(L+i)+\rho\eta\right)},$$
which by using the notation $\omega=L+i,$ can be rewritten as
$$\frac{\omega_1(\omega_2+\eta^2)}{\omega_3(\omega_4+\rho\eta)}\leq\frac{\omega_5(\omega_6+\eta^2)}{\omega_7(\omega_8+\rho\eta)}$$
where $\omega_1=\omega(\omega+2),$ $\omega_2=(\omega+1)^2,$
$\omega_3=2\omega+3,$ $\omega_4=(\omega+1)(\omega+2),$
$\omega_5=(\omega-1)(\omega+1),$ $\omega_6=\omega^2,$
$\omega_7=2\omega-1$ and $\omega_8=(\omega-1)\omega.$ Thus, in order
to show \eqref{ineqross} we need to verify the inequality
$$(\omega_3\omega_5-\omega_1\omega_7)\rho\eta^3+(\omega_3\omega_4\omega_5-\omega_1\omega_7\omega_8)\eta^2+
(\omega_3\omega_5\omega_6-\omega_1\omega_2\omega_7)\rho\eta+\omega_3\omega_4\omega_5\omega_6-\omega_1\omega_2\omega_7\omega_8\geq0,$$
where $L\geq0,$ $\eta\leq0,$ $\rho>0$ and $i\in\{1,2,\dots\}.$
Computations show that for all $L\geq0$ and $i\in\{1,2,\dots\}$ we
have
$$\left\{\begin{array}{l}\omega_3\omega_5-\omega_1\omega_7=-3<0\\
\omega_3\omega_4\omega_5-\omega_1\omega_7\omega_8=(\omega-1)(\omega+2)(8\omega^2+8\omega+3)\geq0\\
\omega_3\omega_5\omega_6-\omega_1\omega_2\omega_7=-2\omega(\omega+1)(2\omega^2+2\omega-1)<0\\
\omega_3\omega_4\omega_5\omega_6-\omega_1\omega_2\omega_7\omega_8=4(\omega-1)\omega^2(\omega+1)^2(\omega+2)\geq0\end{array}\right.,$$
which in turn implies the validity of inequality \eqref{ineqross}.

Now, by using the inequality \eqref{ineqross} we obtain
$${}_1\Delta_{L,\eta}(\rho)\geq\left(1-\frac{C_{L+1,\eta}(\rho)}{C_{L,\eta}(\rho)}\right)F_{L}^2(\eta,\rho),$$
where $L\geq0,$ $\eta\leq0$ and $\rho>0.$ On the other hand for
$\eta\leq0$ and $L\geq0$ the function
$$\rho\mapsto 1-\frac{C_{L+1,\eta}(\rho)}{C_{L,\eta}(\rho)}=1-\frac{(L+2)(2L+1)\sqrt{(L+1)^2+\eta^2}}{(L+1)(2L+3)\sqrt{L^2+\eta^2}}
\frac{L(L+1)+\rho\eta}{(L+1)(L+2)+\rho\eta}$$ is increasing on
$(0,\infty)$ and consequently for all $L\geq0,$ $\eta\leq0$ and
$\rho>0$ we have
$$1-\frac{C_{L+1,\eta}(\rho)}{C_{L,\eta}(\rho)}\geq\lim_{\rho\to0}\left[1-\frac{C_{L+1,\eta}(\rho)}{C_{L,\eta}(\rho)}\right]
=1-\frac{L(2L+1)\sqrt{(L+1)^2+\eta^2}}{(L+1)(2L+3)\sqrt{L^2+\eta^2}}$$
and this together with the above Tur\'an type inequality gives
\eqref{coturan4}.

Finally, let us consider the sharpness of \eqref{coturan4}. By using
the relation \cite[p. 538]{abra}
$$L(2L+1)C_L(\eta)=\sqrt{L^2+\eta^2}C_{L-1}(\eta),$$
we obtain
$$\lim_{\rho\to0}\frac{{}_1\Delta_{L,\eta}(\rho)}{F_{L}^2(\eta,\rho)}=1-\frac{C_{L-1}(\eta)C_{L+1}(\eta)}{C_{L}^2(\eta)}=
1-\frac{L(2L+1)\sqrt{(L+1)^2+\eta^2}}{(L+1)(2L+3)\sqrt{L^2+\eta^2}}$$
and this shows that the above constant (depending on $L$ and $\eta$)
is best possible in \eqref{coturan4}.
\end{proof}

\subsection{Mitrinovi\'c-Adamovi\'c and Wilker type inequalities for
Coulomb wave functions} Now, we present an immediate consequence of the Tur\'an type
inequality \eqref{coturan4}. For this consider the power series
representation of the Coulomb wave function, namely \cite[p.
538]{abra}
$$F_{L}(\eta,\rho)=C_{L}(\eta)\sum_{n\geq0}a_{L,n}\rho^{n+L+1},$$
where
$$a_{L,0}=1,\ a_{L,1}=\frac{\eta}{L+1}\ \ \ \mbox{and}\ \ \ a_{L,n}=\frac{2\eta a_{L,n-1}-a_{L,n-2}}{n(n+2L+1)},\ \ n\in\{2,3,\dots\}.$$
Observe that the Tur\'an type inequality \eqref{coturan4} is
equivalent to
\begin{equation}\label{coturan5}
\mathcal{F}_L^2(\eta,\rho)-\mathcal{F}_{L-1}(\eta,\rho)\mathcal{F}_{L+1}(\eta,\rho)\geq0,
\end{equation}
where $L,\rho>0,$ $\eta\leq0$ and $\mathcal{F}_L(\eta,\rho)$ stands
for the normalized regular Coulomb wave function, defined by
$$\mathcal{F}_L(\eta,\rho)=C_L^{-1}(\eta)\rho^{-L-1}{F}_L(\eta,\rho)=\sum_{n\geq0}a_{L,n}\rho^n.$$

\begin{theorem}\label{thlazar}
If $\eta\leq0,$ $L>-1$ and $0<\rho<x_{L,\eta,1},$ then the following
Mitrinovi\'c-Adamovi\'c and Wilker type inequalities are valid
\begin{equation}\label{lazarineq}
\left[\mathcal{F}_L(\eta,\rho)\right]^{L+\frac{3}{2}}<\left[\mathcal{F}_{L+1}(\eta,\rho)\right]^{L+\frac{5}{2}},
\end{equation}
\begin{equation}\label{wilkerineq}
\left[\mathcal{F}_{L+1}(\eta,\rho)\right]^{\frac{1}{L+\frac{3}{2}}}+\frac{\mathcal{F}_{L+1}(\eta,\rho)}{\mathcal{F}_L(\eta,\rho)}\geq2.
\end{equation}
\end{theorem}

We note that if we choose $\eta=0$ and $L+1/2=\nu$ in Theorem
\ref{thlazar}, then we reobtain the next Mitrinovi\'c-Adamovi\'c and
Wilker type inequalities \cite[Theorem 3]{bariczexpo}
$$\mathcal{J}_{\nu}^{\nu+1}(\rho)\leq\mathcal{J}_{\nu+1}^{\nu+2}(\rho)\ \ \ \mbox{and}\ \ \
\left[\mathcal{J}_{\nu+1}(\rho)\right]^{\frac{1}{\nu+1}}+\frac{\mathcal{J}_{\nu+1}(\rho)}{\mathcal{J}_{\nu}(\rho)}\geq2,$$
where $\nu>-1/2$ and $0<\rho<j_{\nu,1}.$ Here $x_{L,0,n}=j_{\nu,n}$
stands for the $n$th positive zero of the Bessel function $J_{\nu},$
and $\mathcal{J}_{\nu}$ stands for the normalized Bessel function,
defined by $\mathcal{F}_L(0,\rho)=\mathcal{J}_{\nu}(\rho)=2^{\nu}\Gamma(\nu+1)\rho^{-\nu}J_{\nu}(\rho).$
It is important to note here that the above inequalities are valid
for all $\nu>-1$ and the case $\nu=-1/2$ corresponds to the original
Mitrinovi\'c-Adamovi\'c and Wilker inequalities for sine and cosine
functions. See \cite{bariczexpo,basa,wu} for more details on Mitrinovi\'c-Adamovi\'c and Wilker
inequalities.

\begin{proof}[Proof of Theorem \ref{thlazar}]
Consider the function $\varphi_L(\eta,\rho),$ defined by
$$\varphi_L(\eta,\rho)=\left(L+\frac{5}{2}\right)\log\left[\mathcal{F}_{L+1}(\eta,\rho)\right]-
\left(L+\frac{3}{2}\right)\log\left[\mathcal{F}_{L}(\eta,\rho)\right].$$
Observe that the above function is well defined since for each
$\eta\leq0,$ $L>-1$ and $0<\rho<x_{L,\eta,1}$ we have
$$\mathcal{F}_{L}(\eta,\rho)>0\ \ \ \mbox{and}\ \ \ \mathcal{F}_{L+1}(\eta,\rho)>0.$$
Now, by using the recurrence relation \eqref{reccou2} we obtain
\begin{equation}\label{differcou}
\mathcal{F}_{L}'(\eta,\rho)=\frac{\eta}{L+1}\mathcal{F}_{L}(\eta,\rho)-\frac{(L+1)^2+\eta^2}{(L+1)^2(2L+3)}\rho\mathcal{F}_{L+1}(\eta,\rho),
\end{equation}
and consequently
$$2\varphi_{L}'(\eta,\rho)={\eta}\left(\frac{2L+5}{L+2}-\frac{2L+3}{L+1}\right)+
\frac{1}{\mathcal{F}_{L+1}^2(\eta,\rho)}\frac{\rho\mathcal{F}_{L+1}(\eta,\rho)}{\mathcal{F}_{L}(\eta,\rho)}\Phi_L(\eta,\rho),$$
where according to \eqref{coturan5}
\begin{align*}\Phi_L(\eta,\rho)&=\frac{(L+1)^2+\eta^2}{(L+1)^2}\mathcal{F}_{L+1}^2(\eta,\rho)-\frac{(L+2)^2+\eta^2}{(L+2)^2}
\mathcal{F}_{L}(\eta,\rho)\mathcal{F}_{L+2}(\eta,\rho)\\
&\geq
\frac{(L+2)^2+\eta^2}{(L+2)^2}\left[\mathcal{F}_{L+1}^2(\eta,\rho)-
\mathcal{F}_{L}(\eta,\rho)\mathcal{F}_{L+2}(\eta,\rho)\right]\geq0.\end{align*}
On the other hand, by using the Mittag-Leffler expansion \eqref{Comitt}
we obtain
$$\frac{\rho\mathcal{F}_{L+1}(\eta,\rho)}{\mathcal{F}_{L}(\eta,\rho)}=\frac{(L+1)^2(2L+3)}{(L+1)^2+\eta^2}\sum_{n\geq1}
\left[\frac{\rho}{x_{L,\eta,n}(x_{L,\eta,n}-\rho)}+
\frac{\rho}{y_{L,\eta,n}(y_{L,\eta,n}-\rho)}\right]>0,$$ where
$\eta\leq0,$ $L>-1$ and $0<\rho<x_{L,\eta,1}.$ These imply that for
those values of $\eta,\rho,L$ we have $\varphi_{L}'(\eta,\rho)\geq0$
and thus
$$\varphi_L(\eta,\rho)\geq\varphi_L(\eta,0)=0,$$
which completes the proof of \eqref{lazarineq}. Finally, the Wilker
type inequality \eqref{wilkerineq} follows immediately from the
inequality \eqref{lazarineq} and the arithmetic-geometric mean
inequality for the values
$\left[\mathcal{F}_{L+1}(\eta,\rho)\right]^{1/(L+3/2)}$ and
$\mathcal{F}_{L+1}(\eta,\rho)/\mathcal{F}_L(\eta,\rho),$ that is,
$$
\left[\mathcal{F}_{L+1}(\eta,\rho)\right]^{\frac{1}{L+\frac{3}{2}}}+\frac{\mathcal{F}_{L+1}(\eta,\rho)}{\mathcal{F}_L(\eta,\rho)}\geq
2\sqrt{\left[\mathcal{F}_{L+1}(\eta,\rho)\right]^{\frac{1}{L+\frac{3}{2}}}\cdot\frac{\mathcal{F}_{L+1}(\eta,\rho)}{\mathcal{F}_L(\eta,\rho)}}\geq2.$$
\end{proof}

\subsection{Some properties of Coulomb zeta functions} This subsection is devoted to the study of some functions involving the
positive and negative zeros of Coulomb wave functions. We give some basic properties, like recurrence relations, monotonicity properties
and we study the higher order derivatives of these functions. We note that some of the results were already obtained in \cite{stampach}, but here
we use a different approach.

For $s>1$ and $L,\eta\in\mathbb{R}$ let us consider functions
$X_{s,\eta}(L),$ $Y_{s,\eta}(L)$ and $\zeta_{s,\eta}(L),$ which we
call as the Coulomb zeta functions, defined by
$$X_{s,\eta}(L)=\sum_{n\geq1}\frac{1}{x_{L,\eta,n}^{s}},\ \ Y_{s,\eta}(L)=\sum_{n\geq1}\frac{1}{y_{L,\eta,n}^{s}}
\ \ \ \mbox{and}\ \ \
\zeta_{s,\eta}(L)=X_{s,\eta}(L)+Y_{s,\eta}(L).$$ By using the Mittag-Leffler expansion \eqref{Comitt} we obtain for all
$0<\rho<\min\{x_{L,\eta,1},-y_{L,\eta,1}\}$ the generating function
for the Coulomb zeta functions as follows
\begin{align*}
\frac{\rho
F_{L+1}(\eta,\rho)}{F_L(\eta,\rho)}&=\frac{L+1}{\sqrt{(L+1)^2+\eta^2}}\sum_{n\geq1}
\left[\frac{\left(\displaystyle\frac{\rho}{x_{L,\eta,n}}\right)^2}{1-\displaystyle\frac{\rho}{x_{L,\eta,n}}}+
\frac{\left(\displaystyle\frac{\rho}{y_{L,\eta,n}}\right)^2}{1-\displaystyle\frac{\rho}{y_{L,\eta,n}}}\right]\\
&=\frac{L+1}{\sqrt{(L+1)^2+\eta^2}}\sum_{n\geq1}\left[\sum_{m\geq0}\left(\displaystyle\frac{\rho}{x_{L,\eta,n}}\right)^{m+2}+
\sum_{m\geq0}\left(\frac{\rho}{y_{L,\eta,n}}\right)^{m+2}\right]\\
&=\frac{L+1}{\sqrt{(L+1)^2+\eta^2}}\sum_{m\geq0}\left[X_{m+2,\eta}(L)+Y_{m+2,\eta}(L)\right]\rho^{m+2},
\end{align*}
that is, we have
\begin{equation}\label{generator}
\frac{F_{L+1}(\eta,\rho)}{\rho
F_L(\eta,\rho)}=\frac{L+1}{\sqrt{(L+1)^2+\eta^2}}\sum_{m\geq0}\zeta_{m+2,\eta}(L)\rho^{m}.
\end{equation}
Let us suppose that $\eta=0.$ Then
$x_{L,0,n}=-y_{L,0,n}=j_{L+1/2,n}$ for all $n\in\{1,2,\dots\}$ and
the formula \eqref{generator} reduces to
$$\frac{\rho J_{L+3/2}(\rho)}{J_{L+1/2}(\rho)}=\sum_{m\geq0}\zeta_{m+2,0}(L)\rho^{m+2}=
2\sum_{k\geq1}\left[\sum_{n\geq1}\frac{1}{j_{L+1/2,n}^{2k}}\right]\rho^{2k}.$$
Now, let $L+1/2$ be denoted by $\nu,$ then for $|\rho|<j_{\nu,1}$ we
obtain the Kishore's formula \cite[p. 528]{kishore}
$$\frac{\rho J_{\nu+1}(\rho)}{2J_{\nu}(\rho)}=\sum_{k\geq 1}\sigma_{2k}(\nu)\rho^{2k},$$
where
$$\sigma_{2k}(\nu)=X_{2k,0}(\nu-1/2)=\sum_{n\geq1}\frac{1}{j_{\nu,n}^{2k}}$$
is the so-called Rayleigh function. Observe that
$$\lim_{\rho\to0}\left[\frac{F_{L+1}(\eta,\rho)}{\rho
F_L(\eta,\rho)}\right]=\frac{C_{L+1}(\eta)}{C_L(\eta)}=\frac{\sqrt{(L+1)^2+\eta^2}}{(L+1)(2L+3)},$$
and consequently if $\rho\to0$ in \eqref{generator}, then we obtain
$$\zeta_{2,\eta}(L)=\frac{(L+1)^2+\eta^2}{(L+1)^2(2L+3)}.$$
It is also worth to mention that if we use \eqref{generator} and the
power series representation of the Coulomb wave function, then we
obtain
\begin{align*}C_{L+1}(\eta)&\left(a_{L+1,0}+a_{L+1,1}\rho+{\dots}+a_{L+1,n}\rho^n+{\dots}\right)=
\frac{L+1}{\sqrt{(L+1)^2+\eta^2}}C_L(\eta)\\&\times\left(a_{L,0}+a_{L,1}\rho+{\dots}+a_{L,n}\rho^n+{\dots}\right)
\left(\zeta_{2,\eta}(L)+\zeta_{3,\eta}(L)\rho+{\dots}+\zeta_{n+2,\eta}(L)\rho^n+{\dots}\right),
\end{align*}
and identifying the coefficients of $\rho^n$ on both sides we arrive
at the recurrence relation
$$\zeta_{2,\eta}(L)a_{L+1,n}=\sum_{k=0}^na_{L,k}\zeta_{n-k+2,\eta}(L), \ \ \ n\in\{0,1,\dots\}.$$
By using the above relation for $n=1$ we obtain
$$\zeta_{3,\eta}(L)=-\eta\frac{(L+1)^2+\eta^2}{(L+1)^3(L+2)(2L+3)},$$
and other values of $\zeta_{m,\eta}(L)$ can be computed also for
$m\in\{4,5,\dots\}.$ Moreover, by using the relations
\eqref{reccou2} and \eqref{generator} we obtain
$$\frac{F_L'(\eta,\rho)}{F_L(\eta,\rho)}=\frac{L+1}{\rho}+\frac{\eta}{L+1}-\sum_{m\geq0}\zeta_{m+2,\eta}(L)\rho^{m+1}$$
and taking this in \eqref{eqdiffercou} and identifying the
coefficients of $\rho^m$ on both sides we obtain
\begin{equation}\label{recurzeta}(m+2L+3)\zeta_{m+2,\eta}(L)+\frac{2\eta}{L+1}\zeta_{m+1,\eta}(L)=\sum_{k=2}^m\zeta_{k,\eta}(L)\zeta_{m-k+2,\eta}(L),
\ \ \ m\in\{2,3,\dots\}.\end{equation} Observe that the above result
implies that the Coulomb zeta functions are actually rational
functions of $L.$ We mention that the above results were obtained also by \v{S}tampach and \v{S}\v{t}ov\'\i\v{c}ek \cite{stampach}, however, they used a different approach.

Now, we are ready to prove the following new result by using \eqref{recurzeta}.

\begin{theorem}\label{thcouzeta}
If $\eta\leq0$ and $m\in\{2,3,\dots\},$ then the Coulomb zeta
function $L\mapsto \zeta_{m,\eta}(L),$ as well as the functions
$L\mapsto
(m+2L+3)\zeta_{m+2,\eta}(L)+{2\eta}\zeta_{m+1,\eta}(L)/(L+1),$
$L\mapsto \zeta_{m,\eta}(L)/\zeta_{2,\eta}(L)$ and $L\mapsto
(2L+3)^{m-1}\zeta_{m,\eta}(L)$ are completely monotonic on
$(-1,\infty).$
\end{theorem}

Let $\eta=0.$ Then $x_{L,0,n}=-y_{L,0,n}=j_{L+1/2,n}$ for all
$n\in\{1,2,\dots\}$ and
$$\zeta_{s,0}(L)=\sum_{n\geq1}\frac{(-1)^s+1}{(-1)^s}\frac{1}{j_{L+1/2,n}^s}.$$
Observe that for all $s>1$ we have $\zeta_{2s,0}(L)=2X_{2s,0}(L)$
and $\zeta_{2s-1,0}(L)=0.$ Now, taking $m=2r$ in \eqref{recurzeta}
we obtain
$$(2r+2L+3)\zeta_{2r+2,0}(L)=\sum_{k=1}^{r}\zeta_{2k,0}(L)\zeta_{2r-2k+2,0}(L),$$
and if we let $L+1/2=\nu$ and $r+1=q,$ then the above relation
becomes
$$(\nu+q)\sigma_{2q}(\nu)=\sum_{k=1}^{q-1}\sigma_{2k}(\nu)\sigma_{2q-2k}(\nu),$$
which is the result of Kishore \cite[p. 532]{kishore}. We also note
here that in particular when $\eta=0$ the results of Theorem
\ref{thcouzeta} reduce to the main results of Obi \cite[p.
466]{obico} concerning the complete monotonicity of the functions
$\nu\mapsto \sigma_{2q}(\nu),$ $\nu\mapsto
(\nu+1)^q\sigma_{2q}(\nu)$ and $(\nu+q)\sigma_{2q}(\nu)$ on
$(-1/2,\infty),$ where $q\in\{1,2,\dots\}.$

\begin{proof}[Proof of Theorem \ref{thcouzeta}]
Since the sum and product of completely monotonic functions are also
completely monotonic, we have that for $\eta\leq0$ the
functions $L\mapsto\zeta_{2,\eta}(L)$ and $L\mapsto
\zeta_{3,\eta}(L)$ are completely monotonic on $(-1,\infty).$ On the
other hand, from \eqref{recurzeta} we have
$$\zeta_{m+2,\eta}(L)=-\frac{2\eta}{(L+1)(m+2L+3)}\zeta_{m+1,\eta}(L)+\frac{1}{m+2L+3}\sum_{k=2}^m\zeta_{k,\eta}(L)\zeta_{m-k+2,\eta}(L),
\ \ \ m\in\{2,3,\dots\}.$$ Thus, if we suppose that
$L\mapsto\zeta_{s,\eta}(L)$ is completely monotonic on $(-1,\infty)$
for each $s\in\{2,3,\dots,m+1\},$ then by induction we get that
$L\mapsto\zeta_{m+2,\eta}(L)$ is also completely monotonic on
$(-1,\infty).$

Similarly, the functions $L\mapsto (2L+3)\zeta_{2,\eta}(L)$ and
$L\mapsto (2L+3)^2\zeta_{3,\eta}(L)$ are clearly completely
monotonic on $(-1,\infty)$ for all $\eta\leq0.$ Supposing that
$L\mapsto(2L+3)^{s-1}\zeta_{s,\eta}(L)$ is completely monotonic on
$(-1,\infty)$ for each $s\in\{2,3,\dots,m+1\},$ the relation
\begin{align*}(2L+3)^{m+1}&\zeta_{m+2,\eta}(L)=-\frac{2\eta(2L+3)^{m}}{(L+1)(m+2L+3)}\zeta_{m+1,\eta}(L)\\&+
\frac{1}{m+2L+3}\sum_{k=2}^m\left[(2L+3)^{k-1}\zeta_{k,\eta}(L)\right]\left[(2L+3)^{m-k+1}\zeta_{m-k+2,\eta}(L)\right],
\ \ \ m\in\{2,3,\dots\}.\end{align*} and complete mathematical
induction imply that $L\mapsto(2L+3)^{m+1}\zeta_{m+2,\eta}(L)$ is
also completely monotonic on $(-1,\infty).$

Observe that for $\eta\leq0$ the functions
$$L\mapsto \frac{\zeta_{3,\eta}(L)}{\zeta_{2,\eta}(L)}=-\frac{\eta}{(L+1)(L+2)},$$
$$L\mapsto \frac{\zeta_{4,\eta}(L)}{\zeta_{2,\eta}(L)}=\frac{(L+2)(L+1)^2+(5L+8)\eta^2}{(L+1)^2(L+2)(2L+3)(2L+5)}$$
are completely monotonic on $(-1,\infty).$ If $L\mapsto
\zeta_{s,\eta}(L)/\zeta_{2,\eta}(L)$ is completely monotonic on
$(-1,\infty)$ for $s\in\{2,3,\dots,m+1\},$ then in view of
$$\frac{\zeta_{m+2,\eta}(L)}{\zeta_{2,\eta}(L)}=-\frac{2\eta}{(L+1)(m+2L+3)}\frac{\zeta_{m+1,\eta}(L)}{\zeta_{2,\eta}(L)}+
\frac{1}{m+2L+3}\sum_{k=2}^m\frac{\zeta_{k,\eta}(L)}{\zeta_{2,\eta}(L)}\zeta_{m-k+2,\eta}(L),
\ \ \ m\in\{2,3,\dots\}$$ and by using the fact that
$L\mapsto\zeta_{s,\eta}(L)$ is completely monotonic on $(-1,\infty)$
for all $s\in\{2,3,\dots,m\},$ by using mathematical induction we
obtain that $L\mapsto \zeta_{m+2,\eta}(L)/\zeta_{2,\eta}(L)$ is also
completely monotonic on $(-1,\infty).$

Finally, the first part of this theorem together with
\eqref{recurzeta} imply that the function
$$L\mapsto (m+2L+3)\zeta_{m+2,\eta}(L)+{2\eta}\zeta_{m+1,\eta}(L)/(L+1)$$
is also completely monotonic on $(-1,\infty)$ for all
$m\in\{2,3,\dots\}$ and $\eta\leq0.$
\end{proof}

\subsection{Interlacing properties of the zeros of Coulomb wave functions} The first part of the next result is the
extension of a result of Miyazaki et al. \cite[Remark 4.3]{miyazaki}, which states that if $\rho>0,$
$\eta\in\mathbb{R}$ and $L\in\{1,2,\dots\},$ then there is one and only one zero of $\rho\mapsto F_L'(\eta,\rho)$ between two continuous zeros of $\rho\mapsto F_L(\eta,\rho).$

\begin{theorem}
If $L>-1/2$ and $\eta\in\mathbb{R},$ then the zeros of $\rho\mapsto F_L(\eta,\rho)$ and $\rho\mapsto F_L'(\eta,\rho)$ are interlacing. Moreover, if $L>-1$ and $\eta\in\mathbb{R},$ then the zeros of $\rho\mapsto F_L(\eta,\rho)$ and $\rho\mapsto \rho F_L'(\eta,\rho)-(L+1)F_L(\eta,\rho)$ are interlacing.
\end{theorem}
\begin{proof}
In view of \eqref{logder}, for $L>-1$ the function $\rho\mapsto F_L'(\eta,\rho)/F_L(\eta,\rho)$ is decreasing on the interval $(x_{L,\eta,k},x_{L,\eta,k+1}),$ where $k\in\{1,2,\dots\}.$ Moreover, the expression $F_L'(\eta,\rho)/F_L(\eta,\rho)$ tends to $-\infty$ as $\rho\nearrow x_{L,\eta,k+1}$ and tends to $\infty$ as $\rho\searrow x_{L,\eta,k}.$ Since \cite[Remark 17]{stampach} for $L>-1/2$ and $\eta\in\mathbb{R}$ the zeros of $\rho\mapsto F_L'(\eta,\rho)$ are real and simple, it follows that $\rho\mapsto F_L'(\eta,\rho)/F_L(\eta,\rho)$ intersects once and only once the horizontal axis, and the abscissa of the intersection point is actually the $k$th positive zero of $\rho\mapsto F_L'(\eta,\rho).$ The interlacing property of the negative zeros is similar, and thus we omit the details.

Hadamard's theorem states that an entire function of finite order $\tau$ may be represented in the form
$$f(z)=z^me^{P_q(z)}\prod_{n\geq1}G\left(\frac{z}{a_n},p\right),$$
where $a_1,a_2,\dots$ are all nonzero roots of $f(z),$ $p\leq\tau,$ $P_q(z)$ is a polynomial in $z$ of degree $q\leq\tau,$ $m$ is the multiplicity of the root at the origin, and $G(u,p)=(1-u)e^{u+\frac{u^2}{2}+{\dots}+\frac{u^p}{p}}$ for $p>0.$ Combining this with \eqref{infprod} it follows that the growth order $\tau_C$ of the normalized entire Coulomb wave function $\rho\mapsto \mathcal{F}_{L}(\eta,\rho)$ satisfies $1\leq\tau_C<2.$ It is known that the genus of an entire function of order $\tau$ is $[\tau]$ when $\tau$ is not an integer, but the genus of an entire function of natural order $\tau$ can be either $\tau$ or $\tau-1.$ Thus, the normalized entire Coulomb wave function $\rho\mapsto \mathcal{F}_{L}(\eta,\rho)$ is of genus $0$ or $1.$ On the other hand, Laguerre's theorem on separation of zeros states that, if $z\mapsto f(z)$ is an entire function, not a constant,
which is real for real $z$ and has only real zeros, and is of genus $0$ or $1,$ then the zeros of $f'$
are also real and are separated by the zeros of $f.$ According to \cite[Proposition 13]{stampach} when $L>-1$ and $\eta\in\mathbb{R}$ the zeros of $\rho\mapsto \mathcal{F}_L(\eta,\rho)$ are all real. Thus, appealing on Laguerre's separation theorem we conclude that when $L>-1$ and $\eta\in\mathbb{R},$ then the zeros of $\rho\mapsto \rho F_L'(\eta,\rho)-(L+1)F_L(\eta,\rho)$ are all real and are interlacing with the zeros of $\rho\mapsto F_L(\eta,\rho).$
\end{proof}

\end{document}